\documentclass[reqno]{amsart}
\usepackage{amsaddr}
\usepackage{url,amsmath, amsthm, amssymb, amsthm, amsfonts, color, dsfont,graphicx,float} \numberwithin{equation}{section}

\newtheorem{theorem}{Theorem}
 
\newtheorem{lemma}{Lemma}

\begin{document}

\title[Fractional Parts and their Relations to the Zeta Function]{Fractional Parts and their Relations to the Values of the Riemann Zeta Function}
\markright{The Sum of Fractional Parts}
\author{Ibrahim M. Alabdulmohsin}

\address{King Abdullah University of Science and Technology (KAUST),\\
Computer, Electrical and Mathematical Sciences \& Engineering Division\\ 
Thuwal, 23955-6900, Saudi Arabia\\
Tel: +966-5656-77-485\\
Email: ibrahim.alabdulmohsin@kaust.edu.sa}

\maketitle

\begin{abstract}
A well-known result, due to Dirichlet and later generalized by de la Vall\'ee-Poussin, expresses a relationship between the sum of fractional parts and the Euler-Mascheroni constant. In this paper, we prove an asymptotic relationship between the summation of the products of fractional parts with powers of integers on one hand, and the values of the Riemann zeta function, on the other hand. Dirichlet's classical result falls as a particular case of this more general theorem. 
\end{abstract}

\section{Background}

In 1849, Dirichlet established a relationship between the Euler-Mascheroni constant $\gamma = 0.5772\cdots$ and the average of fractional parts. More specifically, writing $[x]$ for the integral (floor) part of the number $x\in\mathbb{R}$ and $\{x\} = x-[x]$ for its fractional part, Dirichlet proved that \cite{pillichshammer2010euler,lagarias2013euler}: 
\begin{equation}\label{dirchlet::first::eq}
\frac{1}{n}\sum_{x=1}^n \Big\{\frac{n}{x}\Big\} = 1-\gamma+ O\Big(\frac{1}{\sqrt{n}}\Big)
\end{equation} 
This surprising connection between $\gamma$ and the average of fractional parts was, in turn, used by Dirichlet to prove that the number of divisors of an integer $n$ is of the order $\log n$. The technique introduced by Dirichlet to prove these results is often called the \lq\lq hyperbola method", which is a counting argument to the number of lattice points that lie beneath a curve \cite{stopple2003primer,lagarias2013euler}. 

The error term in (\ref{dirchlet::first::eq}) is known to be pessimistic. Finding the optimal exponent $\theta>0$ such that: 
\begin{equation*}
\sum_{x=1}^n \Big\{\frac{n}{x}\Big\} - (1-\gamma)n = O(n^{\theta+\epsilon}),
\end{equation*} 
for any $\epsilon>0$ is known as the Dirichlet divisor problem, which remains unsolved to this date. A well-known result of Hardy is that $\theta\ge \frac{1}{4}$, which is conjectured to be the true answer to this problem \cite{lagarias2013euler}. 


In 1898, de la Vall\'ee-Poussin generalized (\ref{dirchlet::first::eq}). He showed that for any integer $w\in\mathbb{N}$:
\begin{equation}\label{poussin::first::eq}
\frac{w}{n}\,\sum_{x=1}^{\frac{n-1}{w}} \Big\{\frac{n}{wx+1}\Big\} =  1-\gamma + O\Big(\frac{1}{\sqrt{n}}\Big)
\end{equation} 
As noted by de la Vall\'ee-Poussin, this result is quite remarkable because the limiting average of the fractional parts remains unchanged regardless of the arithmetic progression that one wishes to use \cite{lagarias2013euler}. 

More recently, Pillichshammer obtained a different generalization of Dirichlet's result. He showed that for any $\beta>1$:
\begin{equation}\label{Pillichshammer::first::eq}
\sum_{x=1}^{\sqrt[\beta]{n}} \Big\{\frac{n}{x^\beta}\Big\} = (1-\gamma_{1/\beta})\sqrt[\beta]{n} + O\big(n^{\frac{1}{\beta+1}}\big),
\end{equation} 
where $\gamma_{1/\beta}$ is a family of constants whose first term is $\gamma_1 = \gamma$ \cite{pillichshammer2010euler}. 

In this paper, we look into a different line of generalizing (\ref{dirchlet::first::eq}). Specifically, we address the question of deriving the asymptotic expressions to summations of the form: 
\begin{equation}\label{eq::fs}
f_s(n) = \sum_{x=1}^n \Big\{\frac{n}{x}\Big\} x^s,
\end{equation}
for non-negative integers $s\in\mathbb{Z}^+$. This is the summation of the products of fractional parts and powers of integers. Clearly, the case where $s=0$ reduces to Dirichlet's classical result. Interestingly, we will show that the asymptotic behavior of this summation is connected to the values of the Riemann zeta function $\zeta(s)$, and we will recover Dirichlet's result in (\ref{dirchlet::first::eq}) as a particular case. More specifically, we prove that:
\begin{equation}
\frac{1}{n^{s+1}}\sum_{x=1}^n \Big\{\frac{n}{x}\Big\}\,x^{s} =  \frac{1}{s} -\frac{\zeta(s+1)}{s+1} + O\Big(\frac{1}{\sqrt{n}}\Big)
\end{equation} 

The Riemann zeta function $\zeta(s)$ is a function of the complex variable defined in $\mathcal{R}(s)>1$ by the absolutely converging series: 
\begin{equation}\label{zeta_def}
\zeta(s) = \sum_{x=1}^\infty \frac{1}{x^s},
\end{equation}
and throughout the complex plane $\mathbb{C}$ by analytic continuation. It is a meromorphic function with a simple pole at $s=1$ with residue 1.

We conclude this section with two classical theorems that we will rely on in our proofs. 
\begin{theorem}[Abel Summation Formula]\label{abel_sum_theorem}
Let $a_x$ be a sequence of complex numbers and $\phi(x)$ be a function of class $\mathbb{C}^1$. Then: 
\begin{equation}
\sum_{1\le x\le n} a_x \phi(x) = A(x)\phi(x) - \int_1^x A(t) \phi'(t) dt,
\end{equation}
where $A(x) = \sum_{k=1}^{[x]} a_x$. 
\end{theorem}

\begin{theorem}[Euler-Maclaurin Summation Formula]\label{theorem::euler_maclaruin}
We have: 
\begin{equation}
\sum_{x=1}^{n} \phi(x)  = C + \int_1^{n} \phi(t) dt + \sum_{k=1}^{s-1} \frac{B_k}{k!} \phi^{(k)}(n) + O(\phi^{(s)}(n)),
\end{equation}
for some constant $C$, where $B_1 = \frac{1}{2}, B_2 = \frac{1}{6}, B_3 = 0, \ldots$ are the Bernoulli numbers. 
\end{theorem} 
These results can be found in many places, such as \cite{HardyDiverg,LampretEM2001}.

\section{Notation} 
We will use the following notation: 
\begin{itemize} 
\item $[x]$ denotes the integral (floor) part of $x$ and $\{x\}=x-[x]$ denotes the fractional part.
\item $\mathbb{N}$ denotes the set of positive integers, often called the \emph{natural} numbers; $\mathbb{Z}^+$ is the set of non-negative integers; $\mathbb{R}$ is the set of real numbers; $\mathbb{C}$ is the set of complex numbers. 
\item $\mathcal{R}(s)$ denotes the real part of $s\in\mathbb{C}$.
\item $|\mathbb{S}|$ denotes the size (cardinality) of the set $\mathbb{S}$. 
\end{itemize} 

\section{The Fractional Transform} 

\subsection{Overview} The key insight we will employ to derive the asymptotic expansion of the function $f_s(n)$ in (\ref{eq::fs}) is that we can solve this problem indirectly by answering a \emph{different} question, first. Specifically, we will be interested in the following function: 
\begin{equation}\label{Phi_eq}
\Phi_s(n) =  \sum_{x=1}^{n-1} \Big[\Big\{\frac{n}{x}\Big\} - \Big\{\frac{n}{x+1}\Big\}\Big] x^s
\end{equation}
More generally, when:
\begin{equation}\label{Phi_eq}
\Phi(n) =  \sum_{x=1}^{n-1} \Big[\Big\{\frac{n}{x}\Big\} - \Big\{\frac{n}{x+1}\Big\}\Big] \phi(x),
\end{equation}
we will call $\Phi(n)$ is the \emph{fractional transform} of $\phi(n)$. The reason $\Phi_s(n)$ allows us to answer our original question is because: 
\begin{align*}
\Phi_s(n) = \sum_{x=1}^{n-1} \Big[\Big\{\frac{n}{x}\Big\} - \Big\{\frac{n}{x+1}\Big\}\Big] x^s = \sum_{x=1}^n \Big\{\frac{n}{x}\Big\} \Big(x^s-(x-1)^s\Big)
\end{align*} 
By expanding the right-hand side using the binomial theorem, we obtain a method of solving our original question.

\subsection{Prelimenary Results}
Next, we present a few useful lemmas related to the fractional transform defined above. Before we do this, we introduce the following symbol: 
\begin{equation}
\partial(n) = \Big\{x\in \mathbb{N}: \;\Big[\frac{n}{x+1}, \frac{n}{x}\Big]\cap \mathbb{N}\neq \emptyset\Big\}
\end{equation} 
In other words, $\partial(n)$ is the set of positive integers that are less than $n$, and for which the interval $[n/(x+1), n/x]$ contains, at least, one integer. For instance, $2\in\partial(5)$ because the interval $[5/3, 5/2]$ contains the integer two, whereas $3\notin \partial(5)$ because the interval $[5/3, 5/4]$ lies strictly between the integer one and the integer two. 

\begin{lemma}\label{fract::transform::value::lemma}
\begin{equation}
\Big\{\frac{n}{x}\Big\} - \Big\{\frac{n}{x+1}\Big\} = \frac{n}{x(x+1)} - \Big|\Big[\frac{n}{x+1}, \frac{n}{x}\Big]\cap \mathbb{N}\Big|,
\end{equation} 
where $|\mathbb{S}|$ denotes the size (cardinality) of the set $\mathbb{S}$. 
\end{lemma} 
\begin{proof}
We have: 
\begin{align*}
\Big\{\frac{n}{x}\Big\} - &\Big\{\frac{n}{x+1}\Big\} = \Big(\frac{n}{x} - \frac{n}{x+1}\Big) - \Big(\Big[\frac{n}{x}\Big]-\Big[\frac{n}{x+1}\Big]\Big)\\
&= \frac{n}{x(x+1)}  - \Big[\frac{n}{x}\Big] + \Big[\frac{n}{x+1}\Big] = \frac{n}{x(x+1)} - \Big|\Big[\frac{n}{x+1}, \frac{n}{x}\Big]\cap \mathbb{N}\Big|
\end{align*}
\end{proof} 
\begin{lemma}\label{fract::transform::value::lemma::2}
If $x\ge \sqrt{n}$ and $x\in\partial(n)$, then: 
\begin{equation*} 
\{\frac{n}{x}\} - \{\frac{n}{x+1}\} = \frac{n}{x(x+1)} - 1
\end{equation*}
\end{lemma} 
\begin{proof}
Because the interval $[n/(x+1), n/x]$ can contain, at most, a unique integer since:
\begin{equation*}
\frac{n}{x} - \frac{n}{x+1} = \frac{n}{x(x+1)} \le \frac{n}{\sqrt{n}(\sqrt{n}+1)} <1
\end{equation*} 
This fact and Lemma \ref{fract::transform::value::lemma} bot imply the statement of the lemma.  
\end{proof} 

\section{Main Results}
We begin with the following lemma: 
\begin{lemma}\label{lemma::3}
For all $s\in\mathbb{R}$:
\begin{equation}
 \sum_{x=1}^{n^\epsilon} \Big[\Big\{\frac{n}{x}\Big\} - \Big\{\frac{n}{x+1}\Big\}\Big] x^s = O(n^{\epsilon s}),
\end{equation}
whose the $O(\cdot)$ constant depends on $s$. 
\end{lemma} 
\begin{proof} 
First, let us consider the following function:
\begin{equation*}
g_n(w) =  \sum_{x=1}^{w} \Big[\Big\{\frac{n}{x}\Big\} - \Big\{\frac{n}{x+1}\Big\}\Big] x
\end{equation*}
Since: 
\begin{equation*}
\sum_{x=1}^{w} \Big\{\frac{n}{x}\Big\} = \sum_{x=1}^{w} \Big\{\frac{n}{x}\Big\} x - \sum_{x=1}^{w} \Big\{\frac{n}{x}\Big\} (x-1) = g_n(w) + \Big\{\frac{n}{w+1}\Big\}w
\end{equation*} 
we obtain:
\begin{equation*}
g_n(w)  = - \Big\{\frac{n}{w+1}\Big\} w + \sum_{x=1}^{w} \Big\{\frac{n}{x}\Big\} = O(w)
\end{equation*}
More precisely, we conclude that: 
\begin{equation*}
-x\le g_n(x)\le x
\end{equation*}
Using this fact and the Abel summation formula in Theorem \ref{abel_sum_theorem}: 
\begin{align*}
 \sum_{x=1}^{n^\epsilon}& \Big[\Big\{\frac{n}{x}\Big\} - \Big\{\frac{n}{x+1}\Big\}\Big] x^s =  \sum_{x=1}^{n^\epsilon} \Big[\Big\{\frac{n}{x}\Big\} - \Big\{\frac{n}{x+1}\Big\}\Big] x x^{s -1}\\
&= n^{\epsilon(s-1)} g_n(n^\epsilon) - (s-1) \int_1^{n^\epsilon} g_n(t)\, t^{s-2} dt\\
&\le  n^{\epsilon s} - (s-1) \int_1^{n^\epsilon} g_n(t)\, t^{s-2} dt\\
&\le  n^{\epsilon s} + |s-1| \int_1^{n^\epsilon} \, t^{s-1} dt \le \Big(1 + \frac{|s-1|}{s}\Big) n^{\epsilon s}
\end{align*}
Similarly: 
\begin{align*}
 \sum_{x=1}^{\sqrt{n}}& \Big[\Big\{\frac{n}{x}\Big\} - \Big\{\frac{n}{x+1}\Big\}\Big] x^s \ge  -\Big(1 + \frac{|s-1|}{s}\Big) n^{\epsilon s}
\end{align*}
Therefore, the statement of the lemma follows. 
\end{proof} 

Now, we are ready to prove our first main result. 
\begin{theorem}\label{thoerem::sum::xs}
For any $s>1$:
\begin{equation*}
\frac{1}{n^s} \sum_{x=1}^{n-1} \Big[\Big\{\frac{n}{x}\Big\} - \Big\{\frac{n}{x+1}\Big\}\Big] x^s = \frac{1}{s-1} -1 + \zeta(s) + O\Big(\frac{1}{\sqrt{n}}\Big)
\end{equation*} 
\end{theorem}
\begin{proof}

We split the sum into two parts: 
\begin{equation}
\Phi_s(n)  = \sum_{x=1}^{\sqrt{n}} \Big[\Big\{\frac{n}{x}\Big\}-\Big\{\frac{n}{x+1}\Big\}\Big] x^s + \sum_{x=\sqrt{n}+1}^{n-1} \Big[\Big\{\frac{n}{x}\Big\}-\Big\{\frac{n}{x+1}\Big\}\Big] x^s
\end{equation} 

The first term is $O(n^{\frac{s}{2}})$ as proved in the previous lemma, which is $O(n^{s-\frac{1}{2}})$ when $s\ge 1$.
Next, we examine the second term. We have by Lemma \ref{fract::transform::value::lemma::2}:
\begin{align*}
 \sum_{x={\sqrt{n}+1}}^{n-1}& \Big[\Big\{\frac{n}{x}\Big\} - \Big\{\frac{n}{x+1}\Big\}\Big] x^s =  \sum_{x={\sqrt{n}}+1}^{n-1} \frac{n\,x^s}{x(x+1)} - \sum_{x=\sqrt{n}+1}^{n-1} \Big|\Big[\frac{n}{x+1}, \frac{n}{x}\Big]\cap \mathbb{N}\Big|\,x^s\\
&=  \sum_{x=\sqrt{n}+1}^{n-1} \frac{n\,x^s}{x(x+1)} - \sum_{x=2}^{\sqrt{n}-1} \Big[\frac{n}{x}\Big]^s \\
&=   \sum_{x=\sqrt{n}+1}^{n-1} \frac{n\,x^s}{x(x+1)} -\sum_{x=2}^{\sqrt{n}-1} \Big(\frac{n}{x}\Big)^s+ \sum_{x=2}^{\sqrt{n}-1} \Big\{\frac{n}{x}\Big\}^s\\
&=   \sum_{x=\sqrt{n}+1}^{n-1} \frac{n\,x^s}{x(x+1)} - n^s \sum_{x=2}^{\sqrt{n}-1} \frac{1}{x^s} + O(\sqrt{n})\\
\end{align*}
Using the fact that for $\mathcal{R}(s)>0$ \cite{limit_zeta_rep}: 
\begin{equation*}
\zeta(s) = \sum_{x=1}^w \frac{1}{x^s} +  \frac{w^{1-s}}{s-1}  -s \int_w^\infty \frac{\{t\}}{t^{s+1}}dt,
\end{equation*}
we conclude that:
\begin{equation*}
 \sum_{x=1}^w \frac{1}{x^s}  = \frac{w^{1-s}}{1-s} + \zeta(s) + O(w^{-s})
\end{equation*} 
Alternatively, the error term $O(w^{-s})$ in the above expression can be derived from the Euler-Maclaurin summation formula in Theorem \ref{theorem::euler_maclaruin}. Hence: 
\begin{align*}
n^s \sum_{x=2}^{\sqrt{n}-1}\frac{1}{x^s} &= n^s \Big( \frac{n^{\frac{1-s}{2}}}{1-s} + \zeta(s)-1 + O(n^{-\frac{s}{2}})\Big) = (\zeta(s)-1)\,n^s + \frac{n^{\frac{1+s}{2}}}{1-s} + O(n^{\frac{s}{2}}),
\end{align*} 

Finally, we look into the remaining term: 
\begin{align*}
n\sum_{x=\sqrt{n}+1}^{n-1} \frac{x^{s-1}}{x+1} &= n \Big[\sum_{x=\sqrt{n}+1}^{n-1} \frac{x^{s-2}}{1 + \frac{1}{x}} \Big] \\
&= n \Big[\sum_{x=\sqrt{n}+1}^{n-1} {x^{s-2}}\big(1- \frac{1}{x} + \frac{1}{x^2} - \cdots\big) \Big] \\
&= n \Big[\sum_{x=\sqrt{n}+1}^{n-1} x^{s-2}- x^{s-3} + \cdots \Big]\\
&= n \Big[\frac{n^{s-1}}{s-1}- \frac{n^{\frac{1}{2}(s-1)}}{s-1} -\frac{n^{\frac{1}{2}(s-2)} + n^{s-2}}{2} \Big] + O(n^{\frac{1}{2}(s-1)} + n^{s-2})\\
&= \frac{n^s}{s-1} +\frac{n^{\frac{1}{2}(s+1)}}{1-s} + O(n^{s-\frac{1}{2}}), 
\end{align*} 
for $s\ge 1$. 
This follows from the Euler-Maclaurin summation formula. More specifically, we have for $u\ge 2$: 
\begin{align*}
\sum_{x=1}^{n-1}& x^{s-u} = C(u) + \frac{n^{s-u+1}}{s-u+1} - \frac{n^{s-u}}{2} + \sum_{k=2}^{m} {{s-u}\choose {k-1}} \frac{B_k}{k} \;n^{s-u-k+1}+ O(n^{s-u-m})\\
\sum_{x=1}^{\sqrt{n}}& x^{s-u} = C(u) +  \frac{n^{\frac{s-u+1}{2}}}{s-u+1} + \frac{n^{\frac{s-u}{2}}}{2} + \sum_{k=2}^{m}{{s-u}\choose {k-1}} \frac{B_k}{k} \;n^{\frac{s-u-k+1}{2}}+ O(n^{\frac{s-u-m}{2}}),\\
\end{align*} 
for some constant $C(u)$ that is independent of $n$. 

Putting everything together, we conclude that for $s>1$: 
\begin{align*}
 \sum_{x=1}^{n-1} \Big[\Big\{\frac{n}{x}\Big\} - \Big\{\frac{n}{x+1}\Big\}\Big] x^s &= \Big[\frac{1}{s-1} -1 + \zeta(s)\Big]\, n^s + O\big(n^{s-\frac{1}{2}}\big),
\end{align*}
which is the statement of the theorem.
\end{proof} 

Theorem \ref{thoerem::sum::xs} is illustrated in Fig \ref{thoerem::sum::xs::figure}. Clearly, this theorem generalizes Dirichlet's result, as promised earlier, because: 
\begin{equation*}
\lim_{s\to 1} \Big\{\frac{1}{s-1} -\zeta(s) \Big\} = -\gamma 
\end{equation*}
and the fact that: 
\begin{equation*}
 \sum_{x=1}^{n-1} \Big[\Big\{\frac{n}{x}\Big\} - \Big\{\frac{n}{x+1}\Big\}\Big]\,x =   \sum_{x=1}^{n-1} \Big\{\frac{n}{x}\Big\} \big(x - (x-1)\big)  = \sum_{x=1}^{n} \Big\{\frac{n}{x}\Big\}
\end{equation*} 

\begin{figure}[t]
\includegraphics[width=\columnwidth]{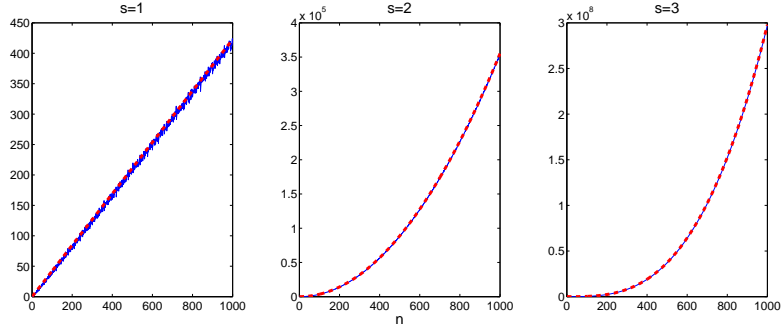}
\caption{A comparison between the values of the fractional transform of $x^s$ (marked in blue) and the asymptotic expression derived in Theorem \ref{thoerem::sum::xs} (marked in red). The $x$-axis is $n$ while the $y$-axis is $\Phi(n)$. The left, middle, and right figures correspond to $s=1,\, s=2$ and $s=3$ respectively.}\label{thoerem::sum::xs::figure}
\end{figure}

Now, we are ready to derive the asymptotic expression of the function $f_s(n)$ given in (\ref{eq::fs}). 
\begin{theorem}\label{theorem::fract::average::zeta}
For all $s\in\mathbb{Z}^+$: 
\begin{equation*}
\frac{1}{n^{s+1}}\sum_{x=1}^n \Big\{\frac{n}{x}\Big\}\,x^{s} =  \frac{1}{s} -\frac{\zeta(s+1)}{s+1} + O\Big(\frac{1}{\sqrt{n}}\Big)
\end{equation*} 
\end{theorem} 
\begin{proof}
Let $f_s(n)$ be as defined in (\ref{eq::fs}). Then, writing by Theorem \ref{thoerem::sum::xs}:
\begin{align*}
\Big[\frac{1}{s-1} +1 &-\zeta(s)\Big] \,n^s + O(n^{s-\frac{1}{2}}) =  \sum_{x=1}^{n-1} \Big[\Big\{\frac{n}{x}\Big\} - \Big\{\frac{n}{x+1}\Big\}\Big] x^s\\
&=  \sum_{x=1}^{n-1} \Big\{\frac{n}{x}\Big\} \,\Big(x^s-(x-1)^s\Big) = \sum_{x=1}^{n-1} \Big\{\frac{n}{x}\Big\}  \sum_{k=1}^s (-1)^{k+1} \,{s\choose k} x^{s-k}\\
&=\sum_{k=1}^s (-1)^{k+1}\,{s\choose k} f_{s-k}(n) = s f_{s-1}(n) + \sum_{k=2}^s (-1)^{k+1}\,{s\choose k} f_{s-k}(n)
\end{align*} 
Hence for $s\in\mathbb{N}$: 
\begin{equation*}
f_{s-1}(n) = \Big[\frac{1}{s-1} -\frac{\zeta(s)}{s}\Big] \,n^s + O(n^{s-\frac{1}{2}}),
\end{equation*} 
which implies the statement of the theorem.
\end{proof} 

\section{Conclusion}
In this paper, we generalized Dirichlet's classical result on the connection between the  Euler-Mascheroni constant and the average of fractional parts. Our theorem reveals that the fractional parts are, in general, connected to the values of the Riemann zeta function $\zeta(s)$. Hence, all of the values of $\zeta(s)$ at positive integers can be expressed as limiting averages of products of fractional parts with the powers of positive integers. 

\bibliography{RH_Proof}
\bibliographystyle{apalike}

\end{document}